\documentclass[a4paper]{article}

\usepackage[english]{babel}
\usepackage[utf8x]{inputenc}
\usepackage[T1]{fontenc}

\usepackage[letterpaper,top=3cm,bottom=2cm,left=3cm,right=3cm,marginparwidth=1.75cm]{geometry}

\usepackage{bm}
\usepackage{graphicx}
\usepackage[colorinlistoftodos]{todonotes}
\usepackage{booktabs}
\usepackage{multirow}
\usepackage{amssymb,amsmath}
\newcommand{\widesim}[2][1.5]{
  \mathrel{\overset{#2}{\scalebox{#1}[1]{$\sim$}}}
}
\numberwithin{equation}{section}
\usepackage{bm}
\usepackage{amsthm}
\newtheorem{lemma}{Lemma}

\newtheorem{proposition}{Proposition}

\newcommand\blfootnote[1]{%
  \begingroup
  \renewcommand\thefootnote{}\footnote{#1}%
  \addtocounter{footnote}{-1}%
  \endgroup
}
\DeclareMathOperator{\Var}{Var}
\DeclareMathOperator{\LCS}{LCS}

\title{Simulations, Computations, and Statistics for Longest Common Subsequences}
\author{Qingqing Liu\thanks{School of Mathematics,
Georgia Institute of Technology,
686 Cherry Street,
Atlanta, GA 30332-0160 USA, qqliu@gatech.edu} \and Christian Houdr\'e\thanks{School of Mathematics,
Georgia Institute of Technology,
686 Cherry Street,
Atlanta, GA 30332-0160 USA, houdre@math.gatech.edu. Research supported in part by the grant $\# 246283$ from the Simons Foundation.}}

\begin{document}
\maketitle
\begin{abstract}
The length of the longest common subsequences (LCSs) is often used as a similarity measurement to compare two (or more) random words.
Below we study its statistical behavior in mean and variance using a Monte-Carlo approach from which we then develop a hypothesis testing method for sequences similarity. 
Finally, theoretical upper bounds are obtained for the 
Chv\'atal–Sankoff constant of multiple sequences. 
%Finally, we briefly study, via simulations, the expectation of the length in the related longest common and increasing subsequences (LCISs) problem for uniform random permutation of $\{1, 2, \cdots, n\}$.
\end{abstract}

%%%%%%%%%%%%%%%%%%%%%%%%%
\section{Introduction}
The study of sequences alignments and  comparisons is an important problem in
bioinformatics and computer science, where a fundamental issue 
is  to compare two or more sequences  and to assess the significance of their similarity or dissimilarity. Within this framework, a general methodology is first to find an optimal alignment of
the sequences and then to compute its score. Afterwards,
some knowledge of the statistics of the alignment score allows to test  hypotheses to tell whether or not the similarity is
significant.

To formalize our discussion, let us introduce our framework.  Following \cite{clote_computational_2000}, 
let $\mathcal{A}$ be a finite alphabet and let
$\texttt{-}\not\in \mathcal{A}$ represent a gap symbol. Let $\Sigma$ to be the set of non-empty
sequences of $\mathcal{A}$, i.e., \ $\Sigma =
\bigcup_{n\ge 0}\mathcal{A}^n$,
where $\mathcal{A}^{0} = \emptyset$ is the empty string. A sequence
$\bm{x}\in\mathcal{A}^n$ has length $n$, denoted by $|\bm{x}|=n$.
Given two sequences
$\bm{a}=(a_1,\ldots,a_n)$ and $\bm{b}=(b_1,\ldots, b_m) \in \Sigma$, we
say that a pair of sequence $\bm{a}^\diamond,\bm{b}^\diamond\in 
%(\mathcal{A}\cup\{\texttt{-}\})^*=
\cup_{n\ge 0} (\mathcal{A}\cup\{\texttt{-}\})^n$ is an alignment of
$(\bm{a},\bm{b})$, if the following three conditions are satisfied: (i) $|\bm{a}^\diamond|=|\bm{b}^\diamond|$, (ii)
$a_i^\diamond\neq \texttt{-}$ or $b_i^\diamond\neq \texttt{-}$ 
for $i=1,\ldots,|\bm{a}^\diamond|$, i.e., no two gaps are aligned, and (iii)
$\bm{a}^\diamond|_\mathcal{A}=\bm{a}$ and $\bm{b}^\diamond|_\mathcal{A}=\bm{b}$, i.e., the restrictions of $\bm{a}^\diamond$ and $\bm{b}^\diamond$ to symbols in $\mathcal{A}$ give respectively $\bm{a}$ and $\bm{b}$.

To measure the similarity of two sequences, assign a score to each
alignment and take the score of the best alignment (i.e., with the highest score) as the similarity score of the two sequences. To define an alignment score, we need a score function 
$s: \mathcal{A}\times\mathcal{A}\to\mathbb{R}$, and a gap penalty function 
$g: \mathbb{N}\to\mathbb{R}$ which is assumed to be subadditive, i.e., 
\begin{equation*}
  \forall \; k,l:\; g(k+l)\le g(k)+g(l).
\end{equation*}

Given a sequence $\bm{u}\in \cup_{n\ge 0} (\mathcal{A}\cup\{\texttt{-}\})^n$, we say that $\bm{u}$
contains a gap of length $k$ at position $i$ 
if $(u_i,\ldots,u_{i+k-1})\in  \cup_{n \ge 0} \{\texttt{-}\}^n $, and there is no other
subsequence of $\bm{u}$ extending   $(u_i,\ldots,u_{i+k-1})$ that is
composed uniquely of \texttt{-}'s. 
Then still following \cite{clote_computational_2000}, $\Delta_k(\bm{u})$ is defined to be the
number of different gaps of $\bm{u}$ having length
$k$, and the score of the alignment
is defined as
\begin{equation}
\label{eq:score}
  s(\bm{a}^\diamond,\bm{b}^\diamond)=
  \sum_{\substack{1\le i \le |\bm{a}^\diamond|\\a_i^\diamond\neq%
      \texttt{-},b_i^\diamond\neq \texttt{-}}} 
  s(a_i^\diamond,b_i^\diamond) 
  - \sum_{1\le k \le |\bm{a}^\diamond|}
  \Delta_k(\bm{a}^\diamond)g(k)
  - \sum_{1\le k \le |\bm{a}^\diamond|}
  \Delta_k(\bm{b}^\diamond)g(k).
\end{equation}

Two types of alignments are commonly used in sequences comparisons, global and local alignment. While a local alignment looks for the segments with best matching scores, the global alignment score corresponds to having as many letter matched as possible in each sequence. 

Although the statistics (mean, variance,
distribution, etc.) of local alignment scores are
well studied \cite{altschul_local_1996,chao_sequence_2009}, there is
still much unknown about the statistics of global alignment
scores. One of the most analyzed global alignment statistics is the length
of the longest common subsequences (LCSs), which is the score of the optimal alignment using the score function
\begin{equation*}
  s(a,b)=
  \begin{cases}
    1,&a=b\\
    0,&a\neq b,
  \end{cases}
  \qquad
\end{equation*}
and a zero gap function, i.e., $ g(k)=0$ for all $k \in \mathbb{N}$. 
Next, given two strings $\bm{a}=(a_1,\ldots,a_n)$ and $\bm{b}=(b_1,\ldots,b_m)$, a sequence $\bm{c}=(c_1,\ldots,c_l)$ is called a common subsequence of $\bm{a}$ and $\bm{b}$ if there exist indices $1\le i_1<i_2<\cdots <i_l\le n$ and $1\le j_1<j_2<\cdots < j_l \le m$ such that $c_k=a_{i_k}=b_{j_k}$ for $k=1,\ldots,l$. Then, the length of the LCS of $\bm{a}$ and $\bm{b}$ is $LCS(\bm{a},\bm{b})=\max\{|\bm{c}|:\,\bm{c} \text{ is a common subsequence of $\bm{a}$ and $\bm{b}$}\}$, and we also use LCS to also represent the length of the common subsequences.
This definition can be naturally extended to the case of three or more sequences, and when the sequences have same length $n$, we denote it by $LC_n$. In the present text, we will only consider the LCSs of sequences of the same length unless otherwise specified.

As far as this paper's content is concerned, we start by summarizing previous studies on the mean behavior of LCS, some notable LCS algorithms and previous work on Monte-Carlo simulation of LCSs.  
We then estimate the variance of the length of LCS of two binary random words using Monte-Carlo experiments (Section~\ref{sec:var}).
Based on these results, and on some recent advances on its limiting distribution \cite{houdre_central_2017}, we build a
hypothesis testing method to test whether two sequences are
significantly similar or not (Section~\ref{sec:lcs-prob}) and conduct extensive Monte-Carlo experiments to determine the parameters of the test (Section~\ref{sec:lcs-alg} to \ref{sec:lcs-expr}). 
Finally, we extend a classical result of \cite{chvatal_upper-bound_1983} valid for two sequences to an arbitrary finite number of sequences (Section~\ref{sec:upper}) and thus obtain new theoretical upper bounds on the Chv\'atal and Sankoff constant in that context.
%Finally, we also briefly study the expectation in a related problem---the length of the longest common and increasing subsequences (LCISs) of random permutations (Section~\ref{sec:lcis}). 

% [[ Detailed descriptions, definitions.... ]]

% Notations:

% $LC(\vec{X_1}, \cdots, \vec{X_m})$: the length of the longest common subsequence of $\vec{X_1}, \cdots, \vec{X_m}$. 

% $\mathbb{E}LC_n$: the expected value of the LCS of length $n$. 

% $\mathcal{A} = \{0, 1, 2, \cdots, k-1\}$: alphabet of size $k$. 

% $\gamma_k^{m*} = \lim_{n \to \infty}\mathbb{E}LC_n /n  $ for $m$ sequences and alphabet of size $k$. 

%%%%%%%%%%%%%%%%%%%%%%%%%
\section{Summary of Previous Work}

\subsection{Theoretical Study}
\label{sec:rel-theory}
The earliest result on the expected length
of LCS is due to Chv\'atal and Sankoff \cite{chvatal_longest_1975}, who proved that the limit
\begin{equation*}
  \gamma_k^*=\lim_{n\to\infty}\frac{\mathbb{E}LC_n}{n}, 
\end{equation*}
exists, where $k$ is the alphabet size, and
the expectation is taken assuming the sequences are
i.i.d. generated, and are also independent of each other.
For uniform binary draws, \cite{chvatal_longest_1975} give bounds for $\gamma_2^*$: $0.727273\le \gamma_2^*\le
0.905118$,
This was followed by many attempts at improving the
bounds---\cite{deken_limit_1979,chvatal_upper-bound_1983,deken_probabilistic_1983,dancik_expected_1994,lueker_improved_2009}, which are summarized in Table~\ref{tab:bounds-gamma-2}. 
\begin{table}[!htp]
\centering
\caption{Theoretical Bounds for $\gamma_2^*$}
\label{tab:bounds-gamma-2}
\begin{tabular}{ l l l}
\toprule
  & lower bound & upper bound \\
\midrule
Chv\'atal and Sankoff&\multirow{2}{*}{0.727273} &\multirow{2}{*}{0.86660}\\
{}\cite{chvatal_longest_1975,chvatal_upper-bound_1983}&&\\
Deken \cite{deken_limit_1979,deken_probabilistic_1983} &0.7615 &0.8575 \\
Dan\v{c}\'ik \cite{dancik_expected_1994} &0.773911 &0.837623\\
Lueker \cite{lueker_improved_2009} &0.788071 & 0.826280\\
\bottomrule 
\end{tabular}

\end{table}

%\subsection{Previous Work on Hypothesis Testing and Monte-Carlo Simulation}
%\label{sec:rel-monte-carlo}
Precise estimates on $\gamma_k^*$, for $2\le k \le 15$, have  been obtained using Monte-Carlo simulations in \cite{boutet_de_monvel_extensive_1999}, and \cite{bundschuh_high_2001} further improves the estimation precision for $k=2,4,8,16$ using a different Monte-Carlo approach. A conjecture on the growth of $\gamma_k^*$, put forward in \cite{sankoff_common_1983}, was positively answered in \cite{kiwi_expected_2005} who showed that $$\lim_{k \to \infty} \sqrt{k}\gamma_k^* = 2.$$

\subsection{Algorithms for LCSs}
\label{sec:rel-alg}
Algorithms to find the best alignments (the ones having the
maximal score) have also been well studied. Since
\cite{needleman_general_1970} developed a dynamic programming algorithm for
global alignment, many improvements or variants have been
developed---\cite{hirschberg_linear_1975} for a linear space
improvement, \cite{smith_identification_1981} for local alignment, 
\cite{gotoh_improved_1982} for affine gap penalty,
\cite{lipman_rapid_1985,altschul_basic_1990,kent_blatblast-like_2002}
for fast heuristic local alignment, and many more.
A detailed review of LCSs algorithms can be found in \cite{bares_algorithms_2009}.

%%%%%%%%%%%%%%%%%%%%%%%%%
\section{Monte-Carlo Simulation for the Variance
\protect\footnote{For all the simulations presented in this paper, the experiments were run on the Partnership for an Advanced Computing Environment (PACE).}}
\label{sec:var}

% There
% are also studies about the variance of the length of LCSs
% \cite{lember_standard_2009,houdre_order_2016} and distribution of the length of LCSs \cite{houdre_central_2017}. 
% It is shown in \cite{lember_standard_2009} that when the two sequences are independent i.i.d. sequences of Bernoulli variables of length $n$, the order of the standard deviation of the length of LCSs is $\sqrt{n}$, when the parameter of the
% Bernoulli  variables  is  small  enough. In \cite{houdre_order_2016}, it is shown that the $r$-th, $1\le r < +\infty$, central moment of the LCSs of two independent random words
% of size
% $n$
% whose letters are identically distributed and independently drawn from a
% finite alphabet is of order $O(n^{r/2})$, when all but one of the letters are drawn with small probabilities. In \cite{houdre_central_2017}, it is shown that for two independent sequences $(X_i)_{i\ge 1}$, $(Y_i)_{i>\ge 1}$ of i.i.d. distributed random variables having the same law and taking their values in a finite alphabet, the length of LCSs satisfies a central limit theorem, under some assumptions of the distribution of $X_1$.
% In \cite{boutet_de_monvel_extensive_1999}, it is observed through Monte-Carlo simulation with $n$ up to 20,000 that the variance of the length of LCSs deviate from $\Theta(n^{2/3})$  (conjectured in \cite{chvatal_longest_1975}). Our simulation shows that when $n$ becomes larger, such deviation also becomes larger and the variance tends to have order $\Theta(n)$.

The theoretical study of the variance of the length of LCSs is less complete. A general linear upper bound has been obtained in \cite{steele_efron-stein_1986}. Lower bounds, also of linear order, have been proved in various biased instances (\cite{lember_standard_2009}, \cite{houdre_order_2016}, \cite{houdre_variance_2016}, \cite{lember_lower_2016}, \cite{gong_lower_2016}, \cite{amsalu_sparse_2016} $\cdots$). But the uniform i.i.d.\ case is still unknown.
In \cite{boutet_de_monvel_extensive_1999}, it is observed through Monte-Carlo simulation, with $n$ up to $20,000$, that the order of the variance of the length of the LCSs of binary random words is at least of order $n^{2\omega'}$, where $\omega' = 0.418 \pm 0.005$. Our simulation shows that when $n$ becomes larger, such deviation also becomes larger and the variance tends to have order $n$.

\subsection{Problem Description}
Given two sequences $\bm{X}=(X_1,\ldots,X_n)$ and
$\bm{Y}=(Y_1,\ldots,Y_n)$ having the same length, where $X_i,Y_i \in \mathcal{A}$ and where again $\mathcal{A}$ is
the alphabet,  we explore, by Monte-Carlo method, the asymptotic behavior of
$\Var LC_n$ when $n$ grows large. 

To perform Monte-Carlo simulations, we need to select an algorithm to compute the length of the LCSs. The dynamic programming algorithm is classical but not efficient enough. Since our experiments are only for $|\mathcal{A}|=2$ or $|\mathcal{A}|=4$, we choose to use the WMMM algorithm \cite{wu_onp_1990}, which is according to \cite{bares_algorithms_2009} very efficient in time and memory when $|\mathcal{A}|$ is small.

\subsection{Experiment Setting}
\begin{itemize}
\item The alphabet size is 2 ($|\mathcal{A}|=2$);
\item For each $n$ we draw 10,000 random sample for Monte-Carlo simulation.
\end{itemize}

\subsection{Experiment Results}

\subsubsection{$\mathbb{P}(X_1=0)=0.5,\; \mathbb{P}(X_1=1)=0.5$}
\label{sec:1vs1}
In this experiment, $n$ ranges from 50,000:50,000:1,000,000.
We plot $\Var LC_n$ against $n$ under a log-log scale in Figure~\ref{fig:varn}.
\begin{figure}[!htp]
  \centering
  \begin{minipage}[b]{0.47\textwidth}
    \centering
    \includegraphics[width=\textwidth]{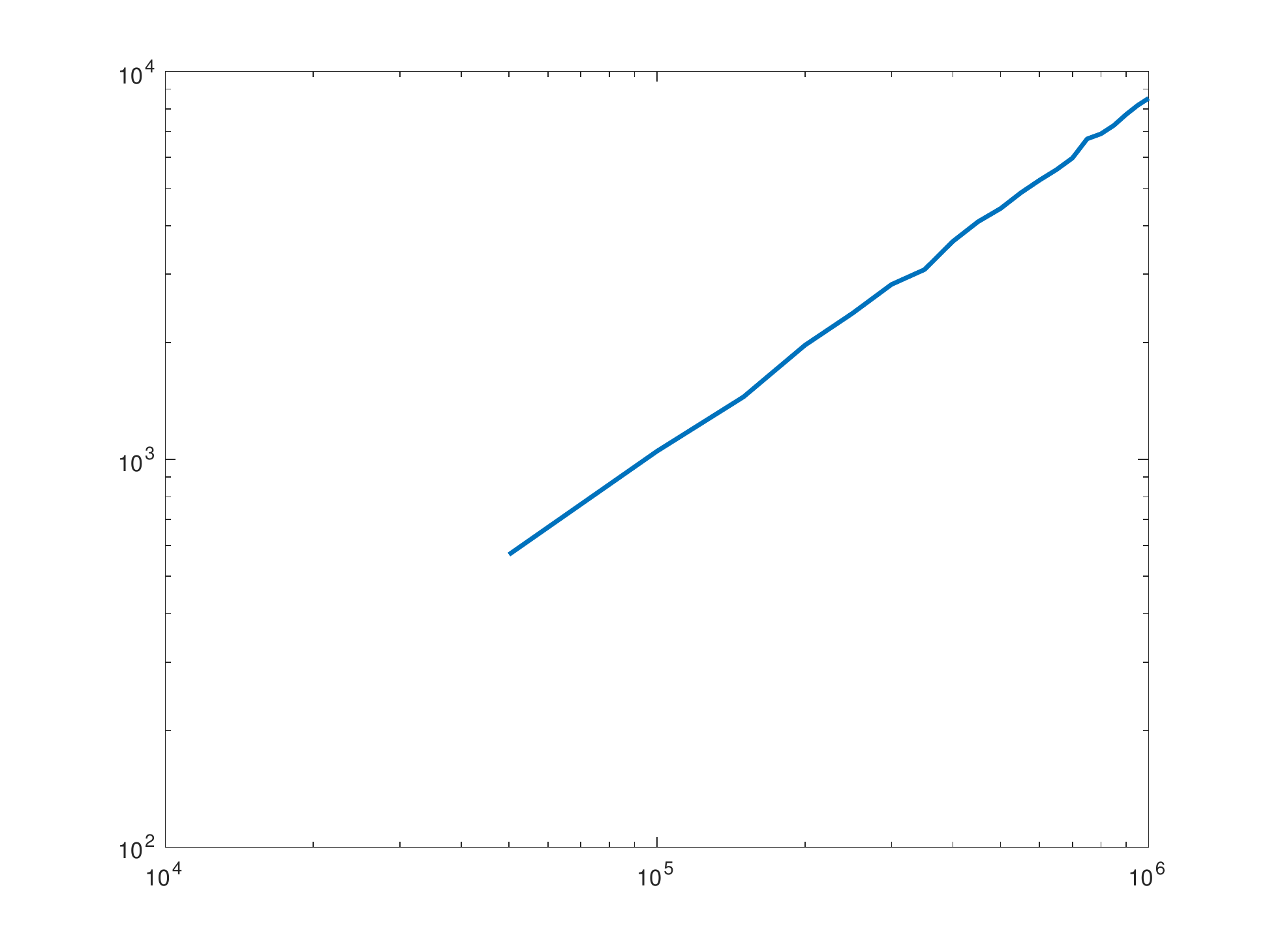}
  \end{minipage}
  \begin{minipage}[b]{0.47\textwidth}
    \centering
    \includegraphics[width=\textwidth]{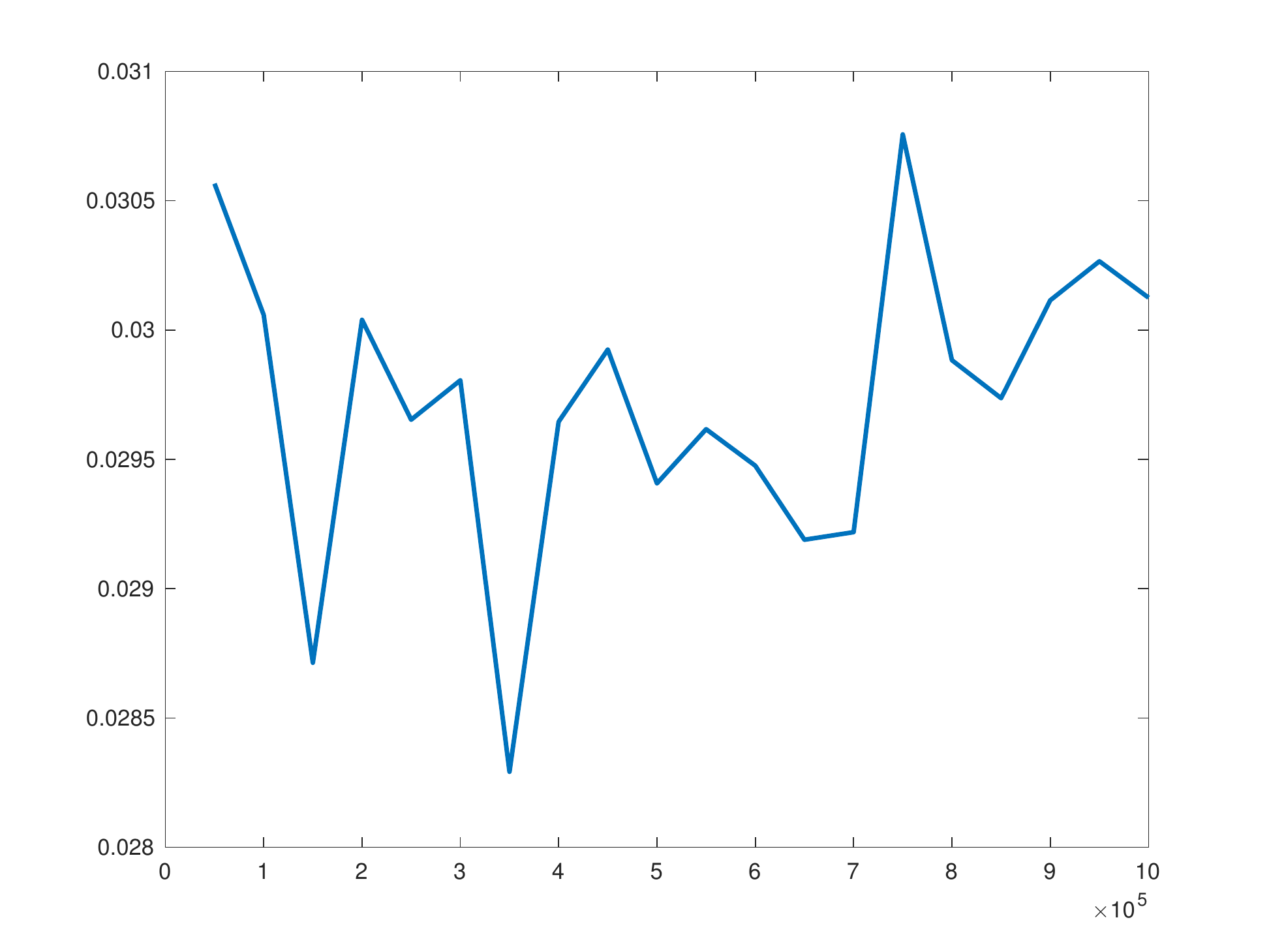}
  \end{minipage}
  \caption{\textbf{Left}: log-log plot of $\Var LC_n$ versus $n$,
    \textbf{Right}: plot of $\Var LC_n/n^{0.9086}$ versus $n$}
  \label{fig:varn}
\end{figure}
\begin{figure}[!htp]
  \centering
  \begin{minipage}[b]{0.47\textwidth}
    \centering
    \includegraphics[width=\textwidth]{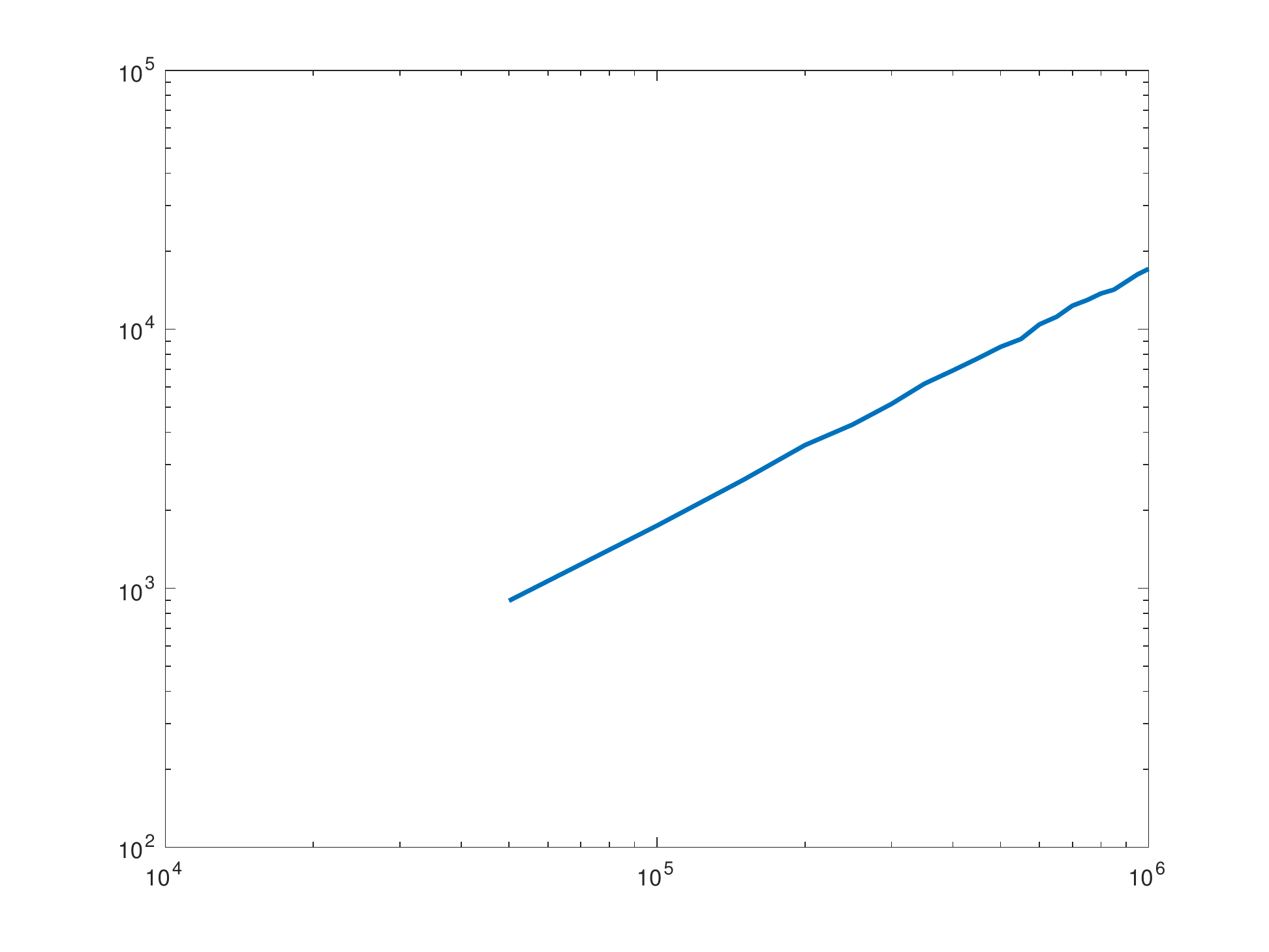}
  \end{minipage}
  \begin{minipage}[b]{0.47\textwidth}
    \centering
    \includegraphics[width=\textwidth]{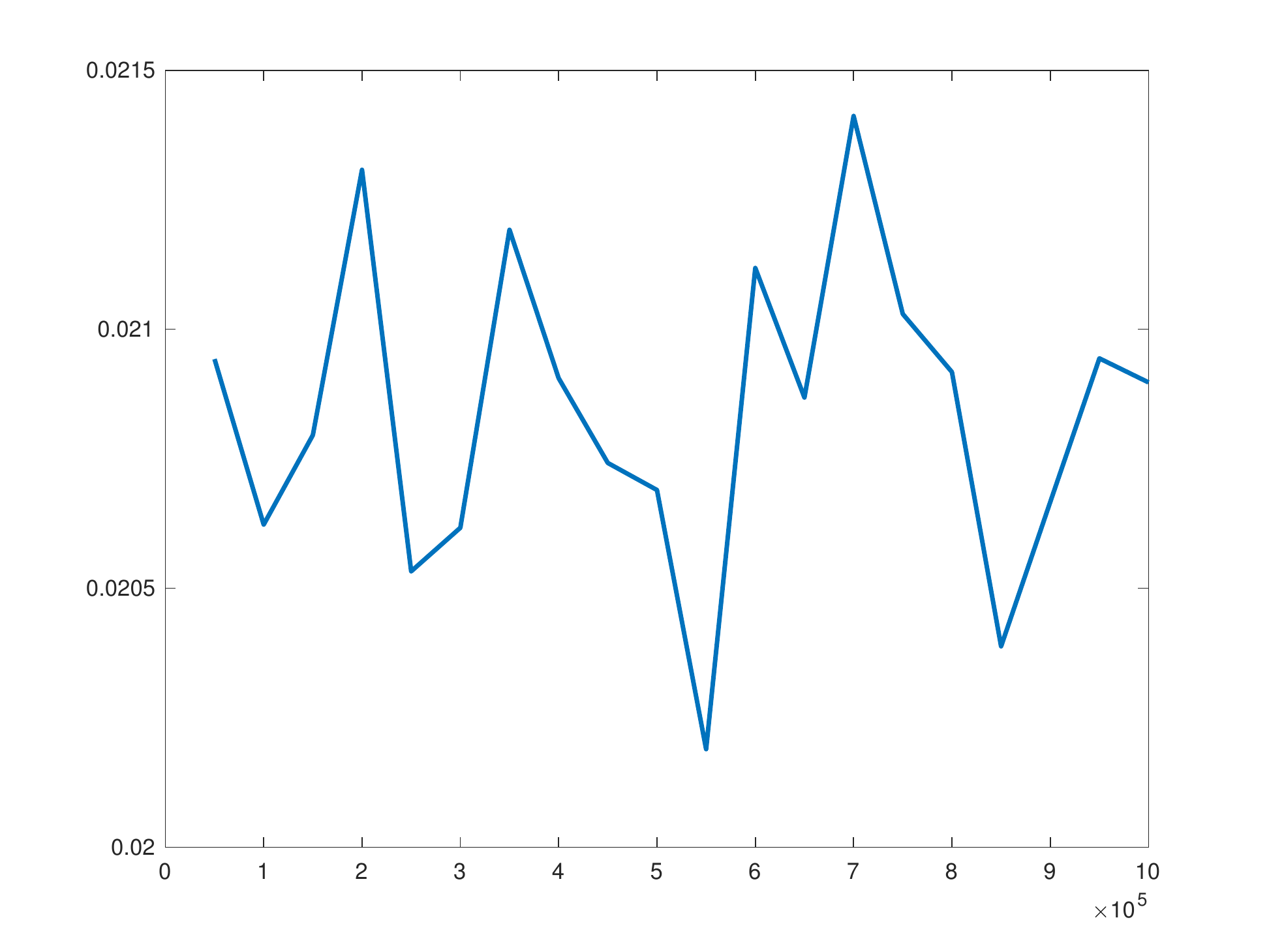}
  \end{minipage}
  \caption{\textbf{Left}: log-log plot of $\Var LC_n$ versus $n$,
    \textbf{Right}: plot of $\Var LC_n/n^{0.9855}$ versus $n$}
  \label{fig:varn_1vs9}
\end{figure}
\begin{figure}[!htp]
  \centering
  \begin{minipage}[b]{0.47\textwidth}
    \centering
    \includegraphics[width=\textwidth]{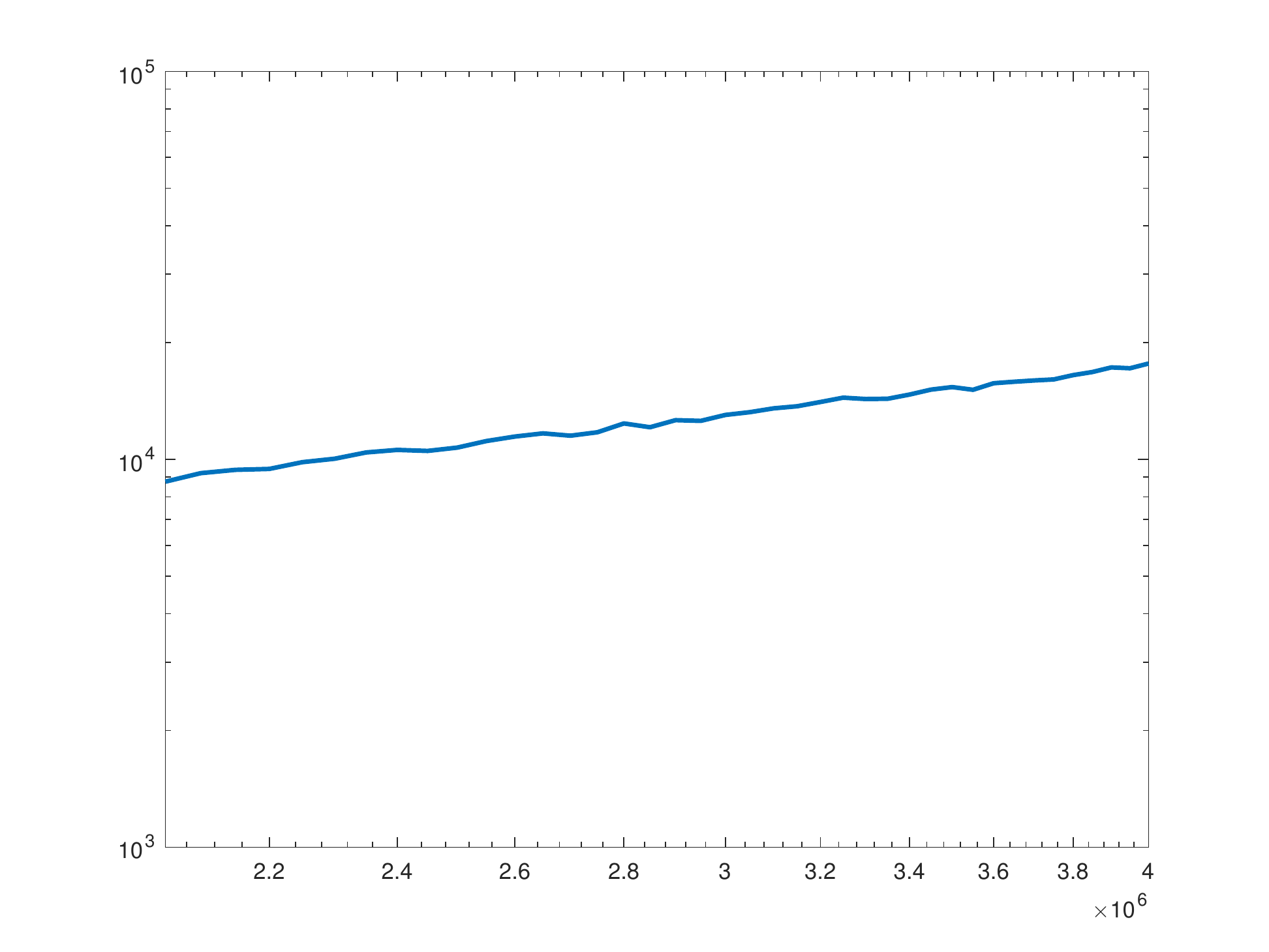}
  \end{minipage}
  \begin{minipage}[b]{0.47\textwidth}
    \centering
    \includegraphics[width=\textwidth]{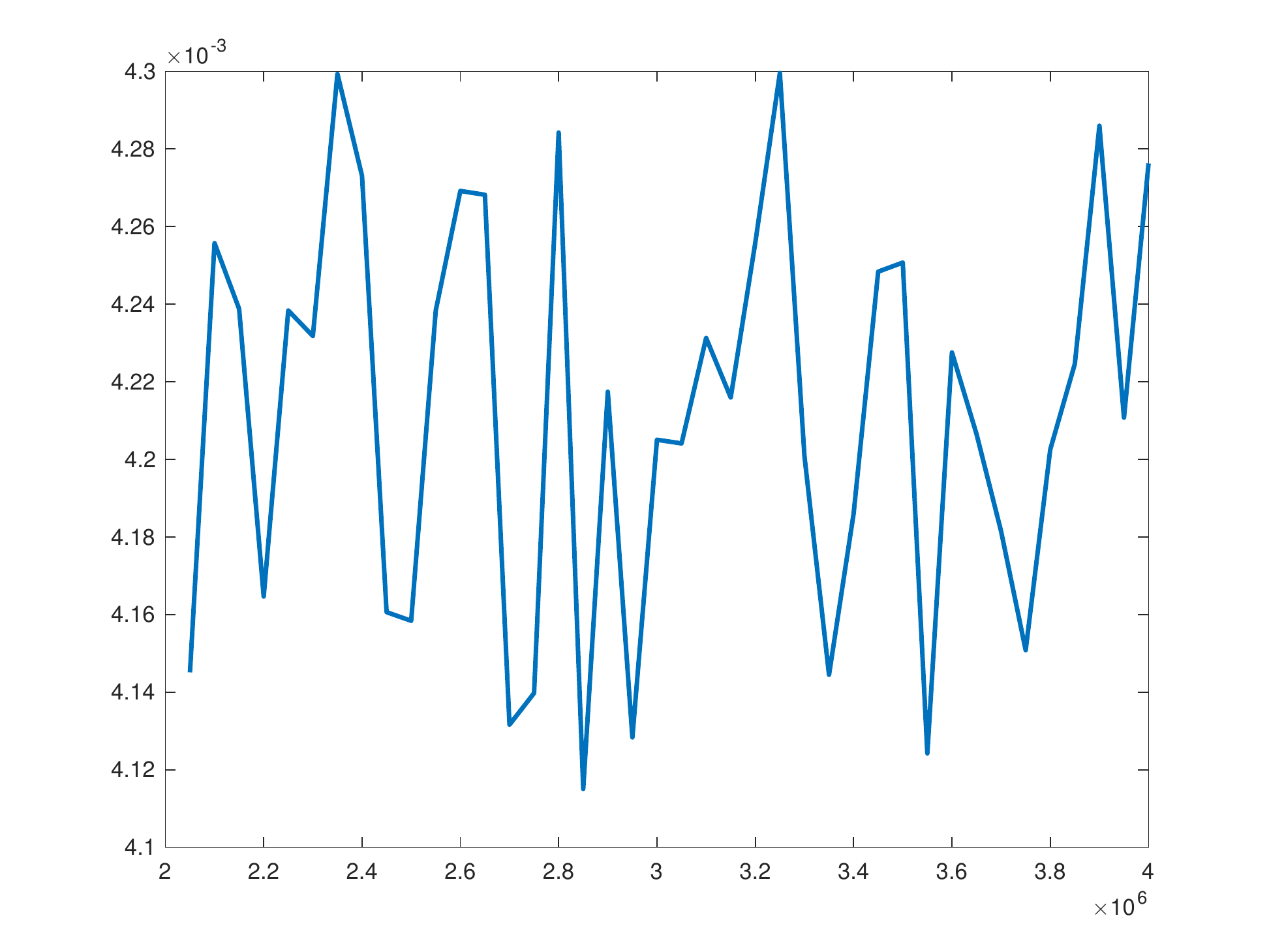}
  \end{minipage}
  \caption{\textbf{Left}: log-log plot of $\Var LC_n$ versus $n$,
    \textbf{Right}: plot of $\Var LC_n/n^{1.0021}$ versus $n$}
  \label{fig:varn_1vs99}
\end{figure}

We found the following relation between $\Var LC_n$ and $n$ using
linear regression
\begin{equation*}
  \Var LC_n\approx 0.0297n^{0.9086}.
\end{equation*}

\subsubsection{$\mathbb{P}(X_1 = 0)=0.1,\; \mathbb{P}(X_1 = 1)=0.9$}
In this experiment, $n$ ranges from 50,000:50,000:1,000,000.
We plot $\Var LC_n$ against $n$ under a log-log scale in Figure~\ref{fig:varn_1vs9}.

We found the following relation between $\Var LC_n$ and $n$ using
linear regression
\begin{equation*}
  \Var LC_n\approx 0.0208n^{0.9855}.
\end{equation*}

\subsubsection{$\mathbb{P}(X_1=0)=0.01,\; \mathbb{P}(X_1=1)=0.99$}
In this experiment, $n$ ranges from 2,050,000:50,000:4,000,000.
We plot $\Var LC_n$ against $n$ under a log-log scale in Figure~\ref{fig:varn_1vs99}.

We found the following relation between $\Var LC_n$ and $n$ using
linear regression
\begin{equation*}
  \Var LC_n\approx 0.0042n^{1.0021}.
\end{equation*}

In all cases, we conjecture that the order of variance of $LC_n$ is:
$$Var LC_n \widesim[3]{asym} cn,$$
where $c$ is a small constant.

%%%%%%%%%%%%%%%%%%%%%%%%%
\section{Hypothesis Testing for the Similarity of two Sequences}

\subsection{Testing Procedure}
\label{sec:lcs-prob}
To test the similarity of two sequences, we propose the following hypothesis testing procedure. Assume we have two sequences $\bm{X}=(X_1,\ldots,X_n)$ and $\bm{Y}=(Y_1,\ldots,Y_n)$, both of length $n$, and then define the null and alternative hypothesis as
\begin{align*}
&H_0: \bm{X}\text{ and } \bm{Y} \text{ are i.i.d. uniformly generated}\\
&H_a: \bm{X}\text{ and }\bm{Y} \text{ have high similarity.}
\end{align*}
Based on the results of \cite{houdre_central_2017}, we use the Z-test and the test statistic is 
\begin{equation}
\label{eq:p}
S=\frac{(LC_n)_{obs}-\mathbb{E}LC_n}{\sqrt{\Var LC_n}},
\end{equation}
where $(LC_n)_{obs}$ is the observed length of the LCS of the two sequences being tested, while $\mathbb{E}LC_n$ and $\Var LC_n$ are the expectation and variance of the length of the LCSs of two sequences, their values estimated by Monte-Carlo simulation.

% \subsection{Related Work}
The paper \cite{reich_statistical_1984} proposed a similarity score based on LCS for comparing two sequences without providing a hypothesis testing procedure, where the estimated LCS statistics were computed for $n$ up to $1000$. Below, we develop a hypothesis testing approach and conduct simulations for $n=10,000$ and extensively verified the effectiveness of the testing method on synthetic sequences. 

\subsection{Experimental Verification}
\label{sec:lcs-alg}

%\subsection{Experiment Design}
\label{sec:lcs-expd}

We conducted several experiments to verify the effectiveness of our testing procedure still using the WMMM algorithm. These
experiments shares the following assumptions/parameters:
\begin{itemize}
\item The alphabet size is 4 ($|\mathcal{A}|=4$);
\item The two sequences $\bm{X}$ and $\bm{Y}$ have the same 
  length ($|\bm{X}|=|\bm{Y}|=n$);
\item The action of inserting a sequence $\bm{Z}$ into another
  sequence $\bm{X}$ is controlled by a parameter $s$. We divide $\bm{Z}$
  into $s$ equally long contiguous segments and $\bm{X}$ into $s+1$ equally long contiguous
  segments, and then insert the $s$ segments from $\bm{Z}$ into
  corresponding positions in the $s$ gaps of $\bm{X}$, as illustrated
  in Figure~\ref{fig:insert}. We denote this action as 
  $\textsc{insert}(\bm{Z},\bm{X},s)$.
\end{itemize}
\begin{figure}[!htp]
  \centering
  \includegraphics[width=0.4\textwidth]{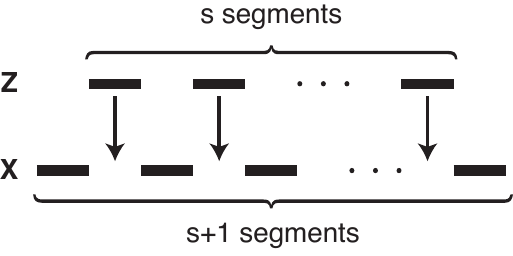}
  \caption{Inserting $\bm{Z}$ into $\bm{X}$.}
  \label{fig:insert}
\end{figure}

%\subsubsection{Determine $\gamma^*_4$ and $c$}
With $n=1,000,000$, we randomly generated 529 pairs of 
$\bm{X}$ and $\bm{Y}$, and compute 
$\gamma_4^*\approx \overline{\LCS(X,Y)}/n\approx 0.654$, 
$c\approx s^2(\LCS(X,Y))/n\approx 0.0075$.

%\subsection{Experiment Results}
\label{sec:lcs-expr}

We use $\alpha=0.05$, $n=10,000$ in our experiments.
For each Monte-Carlo simulation, we draw $10,000$ random samples.

Below are the experiment results.

\subsubsection{Null Hypothesis}
Here $\mathbb{P}(S\le Z_\alpha)=0.9893$, and the histogram of 
$((LC_n)_{obs}-\gamma_4^* n)/\sqrt{cn}$ is in Figure~\ref{fig:hist0}.
\begin{figure}[!htp]
  \centering
  \includegraphics[width=0.6\textwidth]{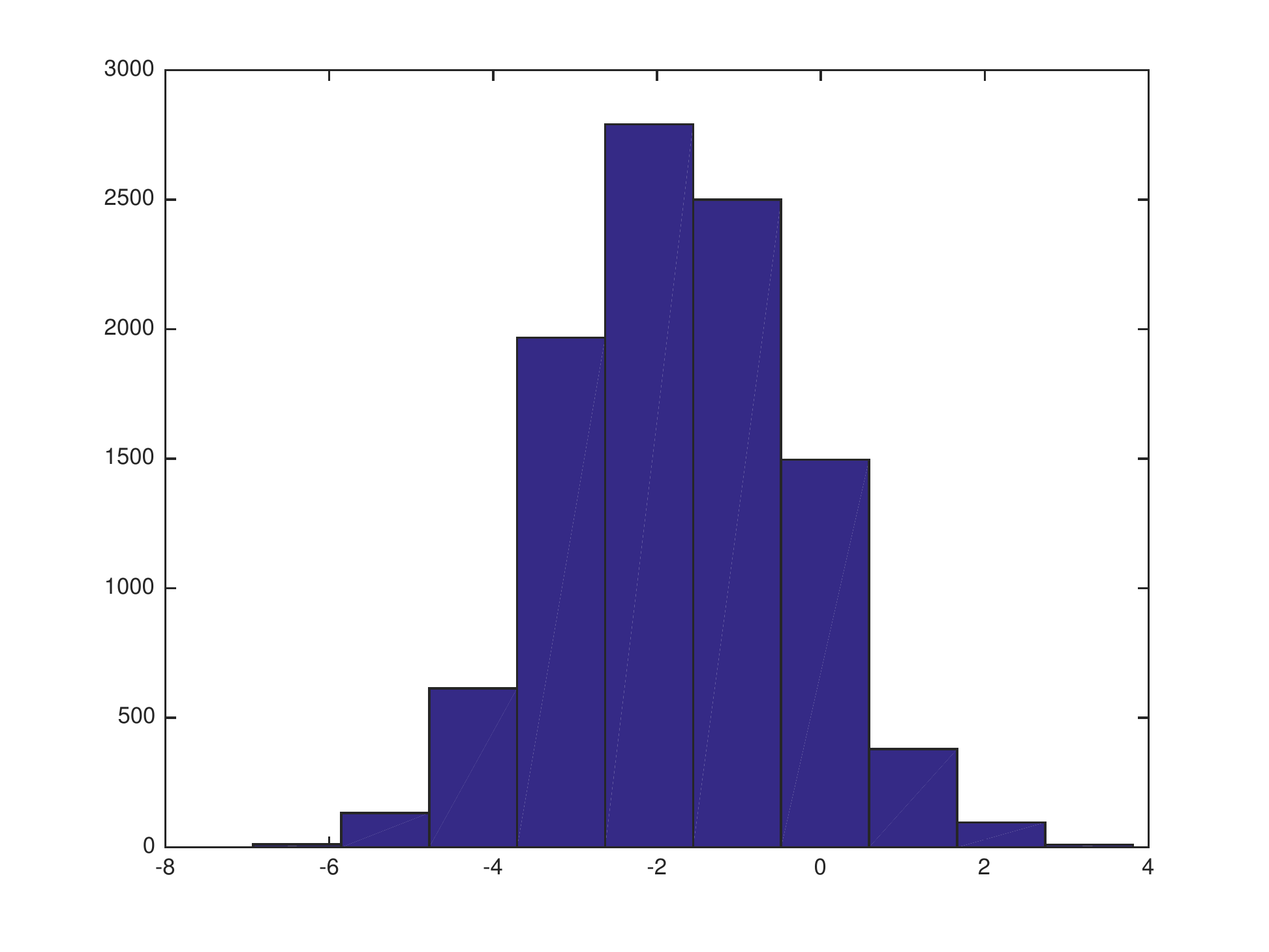}
  \caption{Histogram of $\frac{(LC_n)_{obs}-\gamma_4^* n}{\sqrt{cn}}$}
  \label{fig:hist0}
\end{figure}

\subsubsection{Alternative Hypothesis (1)}
$H_a$: We randomly generated two uniform  i.i.d.~sequences $\bm{X}'$, $\bm{Y}'$
of length $m$, and insert a sequence $\bm{Z}$ of length $n-m$ into
$\bm{X}'$ and $\bm{Y}'$, obtaining $\bm{X}$ and $\bm{Y}$. The results for $p=\mathbb{P}(S\le Z_\alpha)$ are in the following table
\begin{center}
  \begin{tabular}{ccc}
    \toprule
    $m$&$n-m$&$p$\\
    \midrule
    9,000&1,000&0\\
    9,300&700&0.2284\\
    9,350&650&0.4286\\
    9,400&600&0.6119\\
    9,500&500&0.8541\\
    9,900&100&0.9884\\
    \bottomrule
  \end{tabular}
\end{center}

\subsubsection{Alternative Hypothesis (2)}
 $H_a$: We randomly generated two uniform i.i.d.~sequences $\bm{X}'$, $\bm{Y}'$
of length $m=5,000$, and inserted a sequence $\bm{Z}$ of length $n-m=5,000$ into
$\bm{X}'$ and $\bm{Y}'$ obtaining $\bm{X}$ and $\bm{Y}$. The difference is now that each piece of the sequence $\bm{Z}$ has been inserted, with probability 0.8 into both
$\bm{X}'$ and $\bm{Y}'$, with probability 0.1 into $\bm{X}'$ alone, and with probability 0.1 into $\bm{Y}'$ alone. 

In this case, $\mathbb{P}(S\le Z_\alpha) = 0$, and the histogram of 
$((LC_n)_{obs}-\gamma_4^* n)/\sqrt{cn}$ is in Figure~\ref{fig:histar}.
\begin{figure}[!htp]
  \centering
  \includegraphics[width=0.6\textwidth]{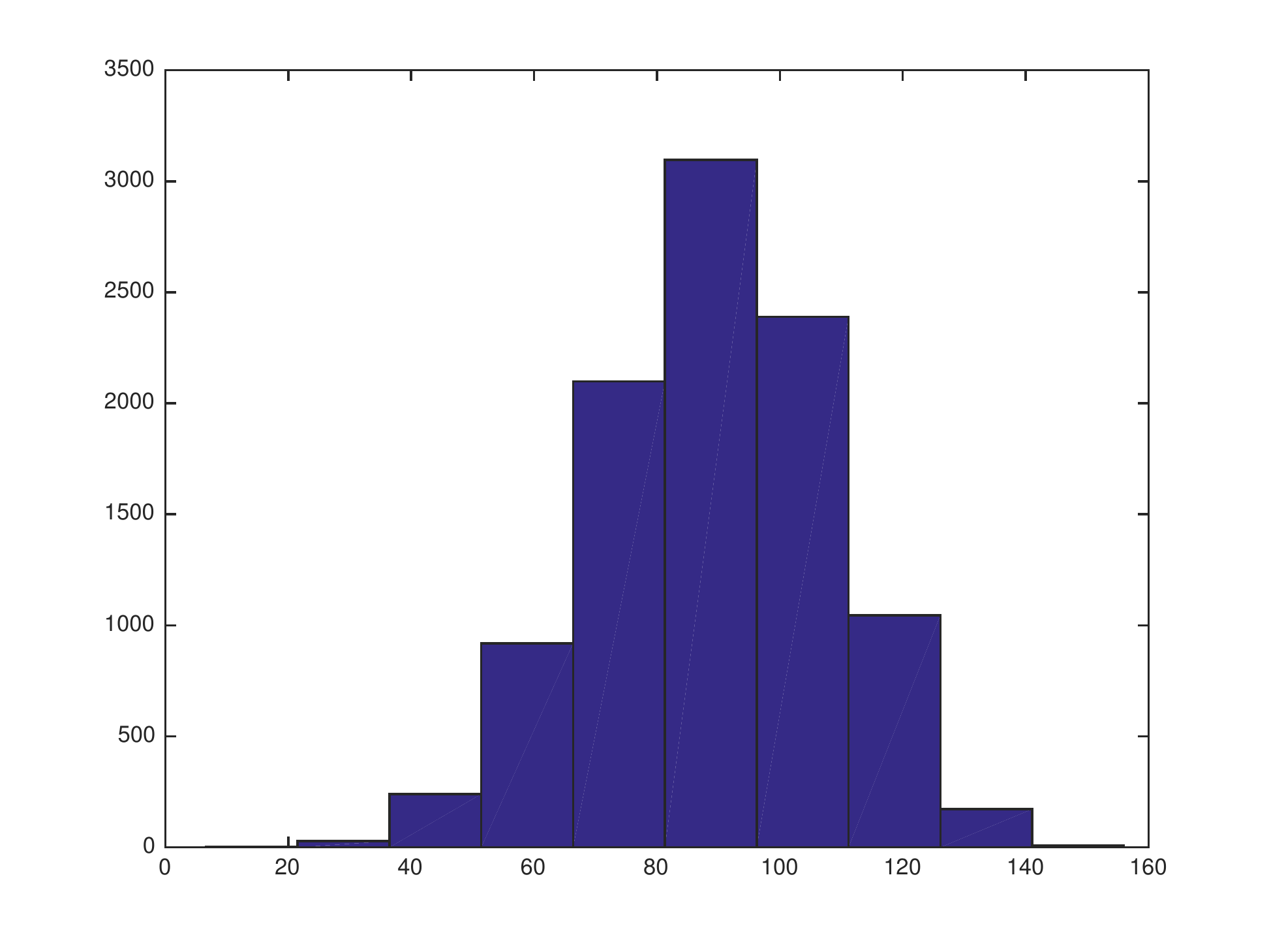}
  \caption{Histogram of $\frac{(LC_n)_{obs}-\gamma_4^* n}{\sqrt{cn}}$}
  \label{fig:histar}
\end{figure}

\subsubsection{Alternative Hypothesis (3)}

 $H_a$: We randomly generated two uniform i.i.d.~sequences $\bm{X}'$, $\bm{Y}'$
of length $m=5,000$, and insert a sequence $\bm{Z}$ of length $n-m=5,000$ into
$\bm{X}'$ and $\bm{Y}'$ obtaining $\bm{X}$ and $\bm{Y}$. This time, each piece of the sequence $\bm{Z}$ was inserted with probability 0.15 into both
$\bm{X}'$ and $\bm{Y}'$, with 
 probability 0.4 into $\bm{X}'$ alone, with  probability 0.4 into $\bm{Y}'$ alone,
and with probability 0.05 into neither $\bm{X}'$ nor $\bm{Y}'$.
\begin{figure}[!htp]
  \centering
  \includegraphics[width=0.6\textwidth]{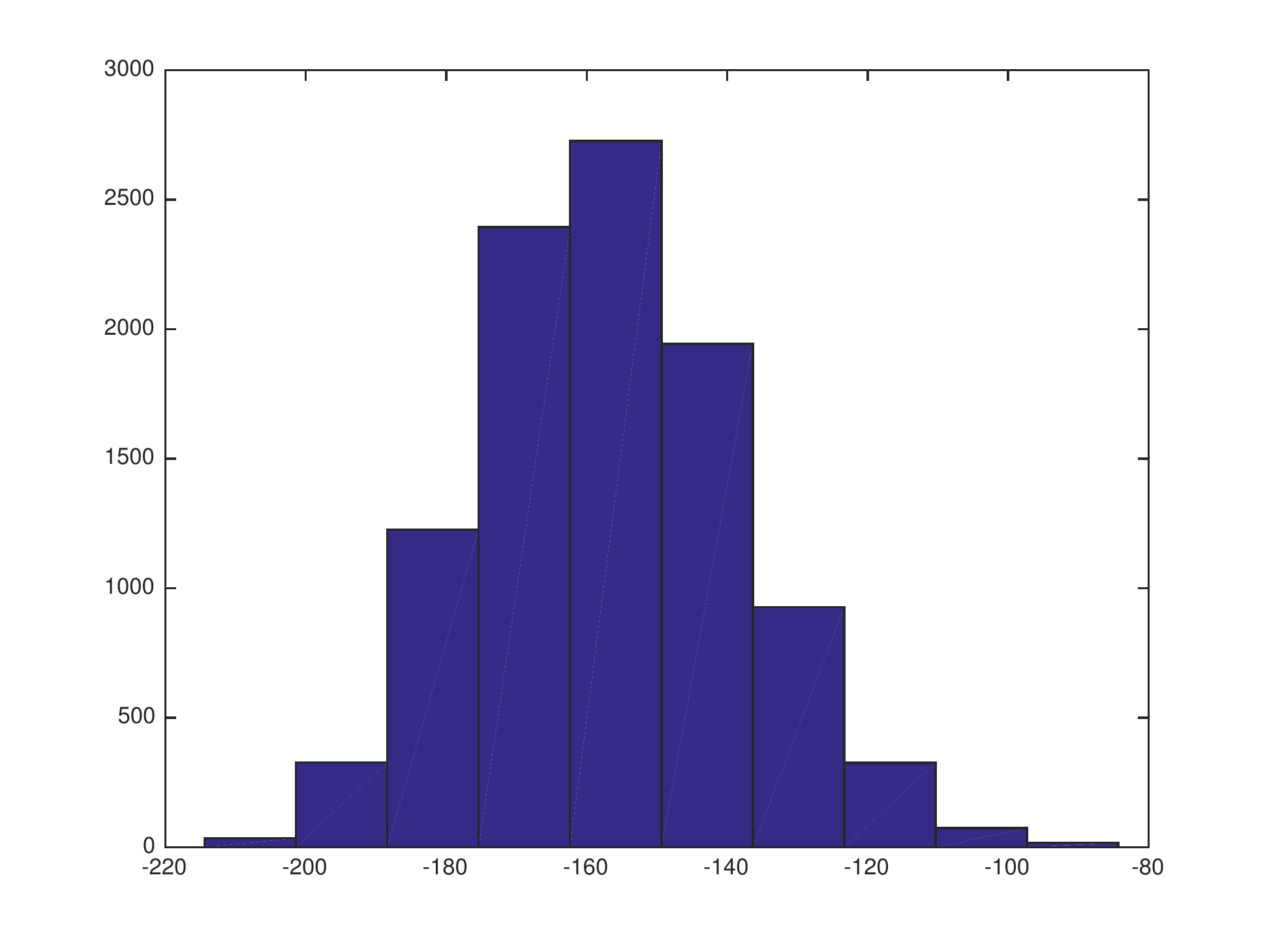}
  \caption{Histogram of $\frac{(LC_n)_{obs}-\gamma_4^* n}{\sqrt{cn}}$}
  \label{fig:histar2}
\end{figure}

In this case, $\mathbb{P}(S\le Z_\alpha) = 1$, and the histogram of 
$((LC_n)_{obs}-\gamma_4^* n)/\sqrt{cn}$ is in Figure~\ref{fig:histar2}.

%\subsection{Experiment Conclusion}
The experiments show that our proposed testing procedure is effective in that the probability $\mathbb{P}(S\le Z_\alpha)$ gets closer to zero when the two sequences have higher similarity.

%%%%%%%%%%%%%%%%%%%%%%%%%
\section{Upper Bound on the Expected Length of LCSs for Multiple Sequences}
\label{sec:upper}
For two sequences and equally likely letters from $\mathcal{A} = \{0, 1, \cdots, k-1\}$, upper bounds on $\gamma_k^*$ are given in \cite{chvatal_upper-bound_1983},   a result which can be extended to an arbitrarily finite number of sequences. 
Below, following \cite{chvatal_upper-bound_1983}, we outline the proof of this extension which will provide upper bounds on $\gamma_{k,m}^{*}$, where $m$ now denotes the number of sequences.

Let $F(n,\bm{s},k)$ be the number of sequences of length $n$ that contains $\bm{s}$, where $\bm{s}$ is any fixed sequence of length $\ell$.  Then a counting and inductive argument developed in   \cite{chvatal_upper-bound_1983} gives:
\begin{lemma}
\label{lem:1}
\begin{equation}
F(n,\bm{s},k) = \sum_{j=\ell}^n \binom{n}{j} (k-1)^{n-j}. 
\label{eq:lem1}
\end{equation}
\end{lemma}

%Using Lemma~\ref{lem:1}, we can obtain an upper bound for $F$. Indeed, letting $\theta = \ell/n > 1/k$, then for $j \ge \ell$,
Since
$${\binom{n}{j +1 }}(k-1)^{n-j-1} \le {\binom{n}{j} }(k-1)^{n-j}, \text{ for } j \ge n/k, $$
\eqref{eq:lem1} leads to
\begin{equation}
\label{neq:F_n_s_k}
F(n,\bm{s},k)  \le n {\binom{n}{\ell} }(k-1)^{n-\ell}, \text{ for }  \ell \ge n/k
\end{equation}

For a fixed $\bm{s}$ of length $\ell$, the number of ordered $m$-tuples of length-$n$ sequences $(\bm{a}_1, \bm{a}_2,\cdots,  \bm{a}_m)$ that all contains $\bm{s}$ as a subsequence is $F^{m}(n,\bm{s},k)$. Then the total number of such $(m+1)$-tuples $(\bm{a}_1, \bm{a}_2,\cdots,  \bm{a}_m, \bm{s})$ is 
\begin{equation*}
G(n,\ell,k) = \sum_{ |\bm{s}|=\ell} F^{m}(n,\bm{s},k),  
\end{equation*}
where the summation is over all the $k^{\ell}$ sequences of length $\ell$. 

Now, let $g(n,\ell,k)$ be the number of $m$-tuples $(\bm{a}_1, \bm{a}_2,\cdots,  \bm{a}_m)$ such that $LC( \bm{a}_1, \bm{a}_2, \cdots, \bm{a}_m )  \ge \ell$, then 
\begin{equation}
\label{neq:g_G}
g(n,\ell,k) \le G(n,\ell,k).
\end{equation}

Next, let $h_k^{(n)}(\theta) $ be the proportion of all ordered $(\bm{a}_1, \bm{a}_2,\cdots, \bm{a}_m)$ such that $LC( \bm{a}_1, \bm{a}_2,\cdots,  \bm{a}_m )  \ge \ell$. 

\begin{lemma}
Let $\theta = \ell/n$, then
$h_k^{(n)} \le (H_k(\theta))^{mn}$, where 
$$H_k(\theta) = \frac{k^{(\theta/m) -1} (k-1)^{1-\theta}}{\theta^{\theta} (1-\theta)^{1-\theta}}.$$
Moreover,  $H_k(\theta)=1$ has a unique solution in the interval $[1/k, 1) $. Let $V_k$ be this solution, then $H_k(\theta)<1$, for $\theta > V_k$.
\end{lemma}

\begin{proof} By Lemma 1 as well as \eqref{neq:F_n_s_k} and \eqref{neq:g_G}, 
\begin{align*}
h_k^{(n)} & = \frac{g(n,\ell,k)}{k^{mn}} \le  \frac{G(n,\ell,k)}{k^{mn}}  
 = \sum_{|\bm{s} | = \ell} \frac{F^m(n,\bm{s},k)}{k^{mn}}   
 \le k^{\ell-mn}  \left\{ n \binom{n}{\ell }(k-1)^{n-\ell} \right\}^m.
\end{align*}
Thus by Stirling's formula,  , 
\begin{align*}
\lim_{n \to \infty} (h_k^{(n)} )^{1/n} 
& \le \lim_{n\to \infty} k^{(\ell-mn)/n}   \left\{ n \binom{n}{\ell }(k-1)^{n-\ell} \right\}^{m/n} \\
& =k^{\theta - m} (k-1)^{m - m \theta}  \lim_{ n \to \infty}   \left\{ n \binom{n}{\ell }\right\}^{m/n}\\
&= k^{\theta - m} (k-1)^{m - m \theta} \frac{1}{\theta^{m\theta}(1-\theta)^{m-m\theta}}\\
&= H_k(\theta)^m.
\end{align*}
To prove the second statement of the lemma, note that   $H_k(\theta) > 0$ for all $\theta \in [1/k, 1)$ and that
$$\lim_{\theta\to 1} H_k(\theta)  = k^{1/m - 1} \lim_{\theta\to 1} \frac{ (k-1)^{1-\theta}}{\theta^{\theta} (1-\theta)^{1-\theta}} = k^{-(m-1)/m}<1,$$
while
$$H_k(1/k) = k^{1/mk} >1.$$
But for $\theta \in [1/k, 1)$, 
$$\frac{d H_k(\theta)}{H_k(\theta)} = \log \frac{(1-\theta)k^{1/m}}{(k-1)\theta}
= \begin{cases}
> 0 &\text { if } \theta > \theta_k \\
<0 & \text{ if } \theta < \theta_k, 
\end{cases}$$
for some $\theta_k$. Therefore, there exists a unique solution $V_k \in [1/k, 1)$, and $H_k(\theta) < 1 $  for $\theta > V_k$.
\end{proof}

Combining the above results leads to:
\begin{proposition}
$$\lim_{n \to \infty} \frac{\mathbb{E}LC_n}{n} \le V_k.$$
\end{proposition}
\begin{proof}
For any $\epsilon > 0$ satisfying $V_k + \epsilon < 1$, separate the total $k^{mn}$ tuples of $(\bm{a_1}, \bm{a_2},\cdots, \bm{a}_m)$ into two categories:  those with longest common subsequences longer than $(V_k +\epsilon)n$, and those with longest common subsequences with length at most $(V_k +\epsilon)n$.  Thus, 
\begin{align*}
\mathbb{E} LC_n & \le (V_k + \epsilon )n \left\{ 1- h_k^{(n)} (V_k + \epsilon )\right\} + (V_k + \epsilon )n  \left\{  h_k^{(n)} (V_k + \epsilon )\right\}  \\
&  \le (V_k + \epsilon )n  + (V_k + \epsilon )n  \left\{  h_k^{(n)} (V_k + \epsilon )\right\}  \\
& \le (V_k + \epsilon)n+ (V_k + \epsilon )n   H_k^{mn}(V_k + \epsilon). 
\end{align*} 
Since $H_k(\theta) < 1 $  for $\theta > V_k$, the last term converges to 0 as $n \to \infty$. 
Thus, 
$$\lim_{ n \to \infty} \frac{\mathbb{E} LC_n}{n} \le V_k + \epsilon, $$
holds for any $\epsilon $ satisfying  $V_k + \epsilon < 1$. 
\end{proof}

Therefore, from the above proposition, $V_k \in [1/k, 1)$ such that $H_k(V_k) = 1$ provides an upper bound on $\gamma_{k,m}^{*}$. In particular, letting $k = 2$, i.e., $\mathcal{A} = \{0, 1\}$, leads to the following table for $\gamma_{2,m}^{*}$, where the lower bounds are obtained in \cite{kiwi_speculated_2009}.

\begin{center}
\begin{tabular}{r|c|c}
number of sequences $m$ & upper bound for $\gamma_{2,m}^{*}$ & lower bound for $\gamma_{2,m}^{*}$ \\
\hline
2 & 0.866595&	0.781281\\
\hline
3 & 0.793026&	0.704473\\
\hline
4 & 0.749082&	0.661274\\
\hline
5 & 0.719527&	0.636022\\
\hline
6 & 0.698053&	0.617761\\
\hline
7 & 0.681605&	0.602493\\
\hline
8 & 0.668516&	0.594016\\
\hline
9 & 0.657797&	0.587900\\
\hline
10& 0.648819&	0.570155\\
\hline
\end{tabular}
\end{center}

The results of \cite{chvatal_upper-bound_1983} have been improved in \cite{deken_probabilistic_1983}.
The current multi-sequence result can similarly be improved using the approach there. In particular, this gives 
for three sequences with binary alphabet, the upper bound 0.791, which is slightly better than $0.793026$ obtained above.
However for four (or more) sequences, even with an alphabet of size 2, this approach becomes rather cumbersome.
Simulation results on $\mathbb{E} LC_n$ are also presented, in some multisequence cases, in \cite{ning_systematic_2013}.

\section{Acknowledgement}
We would like to thank the Partnership for an Advanced Computing Environment (PACE) for providing a high performance infrastructure to simulate and analyze our experimental results, as well as Karim Lounici for his early input and numerous comments on this paper.

%\nocite{bollobas_longest_1988}
\bibliographystyle{plain}
\bibliography{lcs}

\end{document}